\documentclass[12pt,a4paper,leqno]{amsart}
\usepackage{amssymb,latexsym,amstext,amsfonts,amscd,amsmath}
\usepackage[dvips]{graphicx}
\usepackage{eepic}
\usepackage{epsfig}
\allowdisplaybreaks[4]
\catcode`\@=11
\newcommand\lsection{\@startsection {section}{1}{\z@}%
                                   {-3.5ex \@plus -1ex \@minus -.2ex}%
                                   {1.0ex \@plus.2ex}%
                                   {\normalfont\large\bfseries}}
\catcode`\@=12
\newtheorem{thm}{Theorem}[section]

\newtheorem{lem}[thm]{Lemma}
\newtheorem{prop}[thm]{Proposition}
\theoremstyle{definition}
\newtheorem{defn}[thm]{Definition}
\newtheorem{rem}[thm]{Remark}
\newtheorem{exam}[thm]{Example}

\newcommand{\NN}{\mathbf{N}}
\newcommand{\ZZ}{\mathbf{Z}}
\newcommand{\RR}{\mathbf{R}}

\newcommand{\PP}{\mathbf{P}}

\newcommand{\comp}{\text{\scriptsize$\circ$}}
\newcommand{\pd}[2]{\frac{\partial{#1}}{\partial{#2}}}

\newcommand{\dist}{\mathop{{\mathrm{dist}}}}

\newcommand{\RelInt}{\mathop{\mathrm{ri}}}

\newcommand{\cD}{{\mathcal{D}}}
\newcommand{\cO}{{\mathcal{O}}}
\newcommand{\GL}{{\mathop\mathrm{GL}}}
\newcommand{\SO}{{\mathop\mathrm{SO}}}
\newcommand{\OT}{{\mathop\mathrm{O}}}

\title[Nonsmooth inverse mapping theorems]
{Tame nonsmooth inverse mapping theorems}
\author{Toshizumi Fukui}
\address
{Department of Mathematics, Faculty of Science, Saitama
University, 255 Shimo-Okubo, Urawa 338-8570, Japan}
\email{tfukui@rimath.saitama-u.ac.jp}
\author{Krzysztof KURDYKA}
\address
{Universite de Savoie et CNRS, UMR 5127 
Laboratoire de Mathematiques (LAMA)
73-376 Le Bourget-du-Lac cedex FRANCE}
\email{ Krzysztof.Kurdyka@univ-savoie.fr}
\author{Laurentiu Paunescu}                               
\address{School of Mathematics and Statistics\\
University of Sydney, NSW 2006, Australia}
\email{laurent@maths.usyd.edu.au}
\urladdr{}
\date{ 17 December 2007}
\begin{document}
\maketitle

\begin{abstract}
We give several versions of  local and global inverse mapping theorem for tame
non necessarily smooth, mappings. 
Here tame mapping means a mapping which is subanalytic or, 
more generally, definable in some o-minimal structure. Our sufficient conditions are 
formulated in terms of various properties (convexity, positivity of some principal minors, contractiblity)  of the space of Jacobi's matrices at smooth points.
\end{abstract}

\lsection{Introduction.}
The classical inverse mapping theorem gives conditions under which a
$C^r$, $r\ge1$ mapping admits locally a $C^r$ inverse.  
F.~H.~Clarke \cite{Clarke}
generalized the inverse mapping theorem to merely Lipschitz mappings. 
For a Lipschitz map-germ
$f:(\RR^n,0)\to(\RR^n,0)$
he defined the generalized Jacobian $\partial f(0)$ as the convex hull of all matrices which are 
limits of Jacobi's matrices $df(x)$ as $x\to0$. Then he  
 showed that $f$ admits Lipschitz inverse 
when $\partial f(0)$ does not contain singular matrices.

In this paper we show several versions 
of inverse mapping theorems 
for   nonsmooth  mappings, which  belong to 
a tame category. Our results apply also to  smooth  mappings,  
with  assumptions weaker  than the classical ones we obtain existence of  
 not necessarily smooth inverses.
For a convex open subset $U$ of $\RR^m$, 
we say that a mapping $f:U\to\RR^n$ is {\bf tame}, 
if it is subanalytic or, more generally, 
definable in some o-minimal structure. 
We  recall the definition of subanalytic 
mappings and definable mappings
in an o-minimal structure in \S\ref{S:KeyLemma}.
For instance all semi-algebraic maps are tame, 
in particular all rational maps are tame. 

In Section  \ref{S:Example} we  investigate Clarke's idea, of taking the 
convex hull of all
Jacobi's matrices, and we state three global inverse mapping  (more precisely 
global injectivity) Theorems \ref{THM1},\, \ref{THM12} and \ref{THM2}.  
We also 
 state two local inverse Theorems
\ref{THM3} and \ref{THM4} based on a study of  some  minors of Jacobi's 
matrices.  In this section
 we give also several  important and 
relevant examples  which illustrate the relations between our  various 
inverse mapping theorems.
The remaining part of the paper is organized as follows.

In Section \ref{S:Preliminaries} we recall all necessary material from 
convexity theory.
In Section \ref{S:KeyLemma} we state and show a key
property of tame maps which is crucial 
in the proof of our theorems.

In Sections \ref{S:Proof12} and \ref{S:Proof34} 
we prove Theorems \ref{THM1}, \ref{THM12}, \ref{THM2} 
and respectively Theorems \ref{THM3} and  \ref{THM4}.

In Section \ref{S:Contract}, we investigate the contractibility of the set of 
all Jacobi's matrices in the general linear group $\GL(n,\RR)$, as a possible
sufficient condition for the  local invertibility. In the case  $n=2$, precisely we observe the the following fact.  
Let $f:(\RR^2,0)\to(\RR^2,0)$ be a continuous tame mapping  
which is a local  diffeomorphism, except possibly at the origin.
If the mapping 
$$
df:(\RR^2-0,0)\to\GL(2,\RR)
$$
is null homotopic, then $f$ is a homeomorphism. We give also an example which shows 
that this statement is wrong for $n=3$.

The problem of finding sufficient conditions for the injectivity of continuous maps
 has attracted numerous mathematicians working in  different fields.
 In the paper we have essentially cited only a few  relevant contributions.  An excellent   overview 
of related  recent results can be found for
instance in \cite{Gowda} and also in \cite{gutu}.

\lsection{Preliminaries}\label{S:Preliminaries}

\subsection{Convexity}\label{convexity}
We recall here basic notions and facts from the convexity theory, needed 
in this paper.
Let $V$ be a  normed space (of finite
dimension) and
  let $A\ne\emptyset$ be a subset of $V$.
We denote by $af(A)$ the {\bf
affine hull of $A$}, that is, the smallest affine subspace of 
$V$ which
contains $A$. Recall that the {\bf relative interior}, denoted by $\RelInt(C)$,
of a set $C\subset V$ is
the interior of $C$ in its affine hull $af(C)$.
If $C$, a convex set, lies in the both sides of a hyperplane $\pi$, then 
$\RelInt(C)\cap \pi=\RelInt(C\cap \pi))$, we use this several times in our examples.

By ${co}(A)$  we denote the {\bf
convex hull   of $A$}, that is, the smallest convex subset of
$V$ which contains $A$.
By $\overline{co}(A)$  we denote the {\bf
closed convex hull   of $A$}, that is, the smallest closed 
and convex subset of
$V$ which contains $A$. Recall that $\overline{co}(A)=
\overline{co}(\overline A)=\overline{{co}(A)}$ and, if $C$ is convex,
$\RelInt(C)=\RelInt(\overline C)$. In general 
${co}(\overline A) \subseteq \overline{co}(A)$.

Let $A$ be a subset of $V$, we say that a subset $S$ of $A$ is
{\bf extremal} if no point of $S$ is an interior point of a segment with
endpoints in $A$, except the case where both extremities belong to $ S$.
This can be formally stated as follows: for any $x,y\in A$ and any $t\in
(0,1)$, if
$$
tx+(1-t)y\in S
$$
then $x,y \in S$. If $S=\{a\}$ is extremal, then we say that $a$ is an  
{\bf extremal point} of $A$.
 We introduce also a weaker notion of extremality. 
We 
 say that a subset $S$ of $C$ is {\bf semi-extremal} in $C$ if and 
 only if $C\setminus S$ is convex.

We state now several elementary facts about the above sets.

\begin{lem}\label{extremalsubset} Let
$S$ be  an extremal (respectively semi-extremal) subset in $C$ and
 $C'\subset C$. 
Then $S\cap C'$ is an extremal 
(respectively semi-extremal) subset in $C'$.
\end{lem}

\begin{lem}\label{extremalset} Any
extremal  subset  is also a semi-extremal subset. In general the converse is 
not true, it is true for 
one point sets, that is, any semi-extremal point is an extremal point.
\end{lem}

\begin{lem}\label{extremalsetri} If $C \subset V$ is convex and
$S\subsetneq C$ is extremal  in $C$, then $S\cap riC =\emptyset$.
\end{lem}

We need also a lemma about the image of an extremal set.
\begin{lem}\label{extremalimage}  Let $\varphi: V\to W$ be a linear map.
Assume that
$C
\subset V$ is convex and
$S\subset C$ is extremal  (respectively semi-extremal)  in $C$, moreover
$S= C\cap\varphi^{-1}(\varphi(S))$. Then $\varphi(S)$ is extremal 
(respectively semi-extremal) 
in $\varphi(C)$.
\end{lem}
\begin{proof} Indeed $\varphi(C)\setminus \varphi(S)=\varphi(C
\setminus S)$, by our assumption. Hence by Lemma \ref{extremalset}
the set $\varphi(S)$ is extremal
in $\varphi(C)$.
\end{proof}

\subsection{Closed  convex  cones}\label{convexclosedcones}
Let $V$ be a  normed space of finite
dimension and
  let $X$ be a subset of $V$. We say that
  $X$ is a    {\bf cone} if  $x\in X$ and $t\ge 0$ implies $tx\in X$.
For any $A\subset V$ we denote by $\overline{cone}(A)$ the {\bf closed 
convex conic  hull of $A$},
 that is the smallest  closed convex cone in $V$ which
contains $A$.  We denote by $\bar B(1)$  the closed unit ball in $V$. 
We have the following obvious fact.

\begin{lem}\label{generatingclosedcones} Let $X\subset V$ be a cone, then $X$ 
is closed if
and only if $X\cap \bar B(1)$ is compact.
\end{lem}
Indeed the map 
$\RR_{\ge 0} \times ( X\cap \bar B(1))\ni (t,x) \mapsto tx \in  X $  
is surjective and proper.

\begin{lem}\label{extremalconimage}  Let $V,W$ be linear spaces of finite 
dimension and 
let $\varphi: V\to W$ be a linear map.
Assume that
$X
\subset V$ is a closed  convex cone. Then $\varphi(X)$  is also  a closed  
convex cone.
\end{lem}

\begin{proof} It easy to see that $\varphi(X)$ is a convex cone, we prove 
now that it  is also closed.
Indeed  we have 
$$\varphi(X)= \overline{cone}(\varphi (X\cap \bar B(1)),$$
so the claim follows from Lemma \ref{generatingclosedcones}, since 
$\varphi (X\cap \bar B(1))$ is compact.

\end{proof}
\subsection{Extremal sets in the space of  matrices}
\label{extremalmatrices}
We shall use frequently the extremal property with respect to some subspaces 
in the space of matrices.
Let  $X$ be a convex subset of the space of matrices $M(m,n)$.  For a  vector 
$v\in \RR^n$, $|v|=1$ we put
$$
\Sigma_v(X) := \{A\in M(m,n) :\, Av=0\}\cap X.
$$

\begin{prop}\label{extremalprop} The  following conditions are equivalent:
\begin{enumerate} 
\item for each $v\in \RR^n$, $|v|=1$ the set $\Sigma_v(X)$ is extremal (or 
semi-extremal) 
in $X$
\item for any $A, B \in X$ we have
$$\ker (A+B) \subseteq \ker A.$$
\end{enumerate}
\end{prop}
\begin{proof}
First we observe that   $\Sigma_v(X)$ is semi-extremal in $X$ if and only if 
\begin{equation}\label{extreq}
 \varphi (X) \cap \varphi (-X) \subseteq \{0\},
\end{equation}
where $\varphi:M(m,n)\to \RR^n$, $\varphi (A) =Av$.
Now we note that $v\in \ker (A+B)$ if and only if $Av$ and $Bv$ are 
opposite, and obtain the equivalence using inclusion \eqref{extreq}.

\end{proof}

\subsection{Distance to singular matrices}\label{nu}
Let $M(m,n)$ denote the set of $m\times n$ matrices and
 $\Sigma\subset M(m,n)$  the set of singular matrices. We consider  $M(m,n)$ 
as  the space of linear maps from $\RR^m$ to $\RR^n$ equipped with the 
operator norm induced by some fixed norms on  
 $\RR^m$ and $\RR^n$. We first recall two useful facts which we need later.
\begin{lem}\label{LemmaKOS}
Assume that $m\le n$, then for any $A\in M(m,n)$ we have 
$$\nu (A):=\dist(A,\Sigma)=\inf\{|A v|:v\in S^{m-1}\}.$$ 
Moreover if $m=n$, then $\nu (A)= \Vert A^{-1}\Vert^{-1}$.
\end{lem}
\begin{proof} See Proposition 2.2 in \cite{KOS}.  
\begin{rem}
Usually for  given matrix  $A$  the computation of $\nu (A)=\dist(A,\Sigma)$ 
is not obvious. Often
it is enough and more efficient to use an auxiliary function $g$ with the 
property that, there exist constants
$a,b>0$ such that 
$ag(A) \le  \nu (A) \le bg(A)$, for any matrix $A$. For instance we can take
$$\displaystyle
g(A)= \text{max}_I \frac {|\det A_I|}{(\sum_{i,j}A_{I,i,j}^2)^{1/2}}\, ,
$$
where $A_I$ are the $m\times m$ matrices extracted from $A$, and $A_{I,i,j}$ 
denotes 
the determinant (minor) of the matrix $A_I$ obtained  from $A$ by deleting
the $i$-th row  and the $j$-th column. 
\end{rem}

\end{proof}

For further reference we mention here the classical Theorem of the Invariance 
of Domain.

\begin{lem}[Invariance of Domain]\label{InvarianceOfDomain}
Let $U$ be an open subset of $\RR^n$.
Then, every injective continuous mapping $f:U\to\RR^n$ is open.
\end{lem}
\begin{proof}
See {\cite[IV. Proposition 7.4]{Dold}}. 
\end{proof}

\subsection{Notations}\label{Notations}
Let $U$ be an open subset of $\RR^m$ and  $f:U\to\RR^n$ be a continuous tame 
mapping. 
We denote by $B(f)$ the set of points where $f$ is not differentiable. 

Let $V$ be an open subset of $U$,  we denote  by

$${co}(f,V): ={co}(\cD(f,V)),\, \,\overline {co}(f,V):=
\overline{co}(\cD(f,V))$$ 
the convex hull and the  closed convex  hull respectively, of the set 
$$
\cD(f,V):=\{df(x):x\in V-B(f)\}.
\quad\textrm{} 
$$
For short   we  shall  write ${co}(f)$ and $\overline{co}(f) $ 
instead 
of ${co}(f,U)$
and of $\overline {co}(f,U)$ respectively. Also we use
 $\cD(f)$ instead of $\cD(f,U)$, and for a given  $v \in \RR^n$, we put
 $\cD_v(f):=\{Av:A\in\cD(f)\}$.
 Finally we consider  the closed convex cone generated by   $\cD(f)$, and write
 $$\overline {cone}(f) : = \overline{cone}(\cD(f)).$$
Clearly $\overline {cone}(f)  = \overline{cone}({co} (f)) =
\overline{cone}(\overline {co} (f)) $.

The following routine property (Koopman-Brown theorem cf. \cite{vandebDries}) 
of tame sets will be useful. 
\begin{lem}\label{omiting}
Let  $B \subset \RR^m$ be a  set definable in an o-minimal structure. Assume 
that $B$ is nowhere dense
in $\RR^m$. Then for each $x\in \RR^n$ the set 	 
$B(x) = \{y\in \RR^m:\,  [x,y]\cap B \text{ is finite}\}$ is dense  and 
definable in $\RR^m$.
\end{lem}

\lsection{Results and Examples}\label{S:Example}
Let us first state the result of F.Clarke \cite{Clarke}  which has 
inspired our  work.  
Recall that if $U\subset \RR^n$ is open  and  $f:U\to\RR^n$ is a Lipschitz 
map,  then $f$ is almost
everywhere differentiable, hence the set the $U\setminus B(f)$ is dense in 
$U$.
 We  define $\overline {co}(f)$ as in Section \ref{S:Preliminaries}.

\begin{thm}[Clarke]\label{THMClarke}
Let  $U$ be a convex open subset of $\RR^n$ and let  $f:U\to\RR^n$  
be a Lipschitz map. 
Assume that  $\overline {co}(f) \cap \Sigma=\emptyset$, then 
$f$ is injective, and  moreover  $f^{-1}$ is 
Lipschitz. 
\end{thm}
Our  first global injectivity theorem is the following.

\begin{thm}\label{THM1}
Let  $U$ be a convex open subset of $\RR^m$. 
Assume that $f:U\to\RR^n$, $m\le n$, is tame and continuous. 
If $\dist({co}(f),\Sigma)\ge\delta$, then 
$$
|f(x')-f(x)|\ge\delta|x'-x|\qquad\textrm{for all $x$, $x'\in U$.}
$$
In particular $f$ is injective. If $m=n$, then  $f(U)$ is open, 
$f: U\to f(U)$ is a homeomorphism and $f^{-1}$ satisfies the
Lipschitz condition  with constant $\delta^{-1}$. 
\end{thm}

 Note that Theorem \ref{THM1}  implies Theorem \ref{THMClarke}. Indeed, 
if $f $ is Lipschitz
the set $\overline {co}(f) $ is compact, hence the condition 
$\overline {co}(f) \cap \Sigma=\emptyset$ implies that
there exists $\delta>0$ such  that $\dist({co}(f),\Sigma)\ge\delta$.
In  the category of tame maps  Theorem \ref{THM1}  is  stronger 
than Clarke's theorem, since
we do not assume that $f$ is Lipschitz or even locally Lipschitz.  Actually 
both theorems can be seen as
non-smooth  variants (with  Lipschitz inverse) of the celebrated Hadamard's 
theorem, see \cite{Hadamard}, for $C^1$ maps.

In the next theorem we will assume that that Clarke's type assumption is 
verified only generically.

\begin{thm}\label{THM12}
Let $U$ be a convex open subset of $\RR^n$ and $f:U\to\RR^n$  a continuous
tame mapping. Assume that there exists a nowhere dense closed $B\subset U$, 
such that
$f$ is a $C^1$  immersion  on $U\setminus B$, that is,
$df(x)$ is non-singular for any $x\in U\setminus B$. Suppose that $f$ 
satisfies 
the following conditions:
\begin{itemize}
\item[$(C_r)$]  $
\Sigma\cap {co}(f,U\setminus B) = \emptyset
$, 
where $ {co}(f,U\setminus B) := co \{df(x):x\in U\setminus B\}$;
\item[(I)] $f$ is injective on $B$.
\end{itemize}

Then $f$ is injective.
\end{thm}

Surprisingly, in the next theorem we allow some singular matrices
in $\overline{co}(f)$.
For any vector $v\in \RR^n$, $|v|=1$,  we put
$$
\Sigma_v(f): = \{A\in M(m,n) :\, Av=0\}\cap \overline {co}(f).
$$
\begin{thm}\label{THM2}
Let $U$ be a convex open subset of $\RR^m$.
A  locally Lipschitz tame mapping $f:U\to\RR^n$ is injective
if it satisfies the following conditions:
\begin{itemize}
\item[$(C_e)$] for any   $v\in \RR^n$, $|v|=1$ the set
$\Sigma_v(f) $
is extremal in $\overline {co}(f)$,
\item[(S)] $f$ is not constant on any segment in $U$.
\end{itemize}
\end{thm}

\begin{rem}\label{remarkTHM2}
Note that a priori we do not assume that there is at least one
point in $U$ at which the Jacobi matrix of $f$ is invertible.
\end{rem}

\begin{exam}\label{ExampleCubic}
Consider the homeomorphism $f:\RR^2\to\RR^2$ defined by
$(x,y)\mapsto(x^3,y^3)$.
The Jacobi's matrices are given by
$\begin{pmatrix}3x^2&0\\0&3y^2\end{pmatrix}$.
We observe that $\overline {co}(f)$ is the set of all matricies
$\begin{pmatrix}a&0\\0&b\end{pmatrix}$, where $a,b\ge 0$.
For $v=(\pm 1,0)$ and $v=(0,\pm1)$ the corresponding sets $\Sigma_v(f)$ are
the closed halflines in the boundary of $\overline {co} (f) $, so they are 
extremal
in $\overline {co} (f)$. For any other $v$ the set  $\Sigma_v(f)$ is just the 
matrix
$A=0$ which is extremal in $\overline {co} (f)$. Hence the condition $(C_e)$ is
satisfied.
\end{exam}
It seems rather difficult to weaken the assumptions in Theorem \ref{THM2}. To 
see this let us consider the following examples.
\begin{exam}\label{Exampleantirelint}
Consider a continuous  map $f:\RR^2\to\RR^2$ defined by
$f(x,y)= (x+y^3,x)$ for $x\ge 0$ and
$f(x,y)= (x+y^3,0)$ for $x\le 0$.  Clearly this map is not 
injective.  
The Jacobi's matrices  of $f$ are of the form
$\begin{pmatrix}1&a\\1&0\end{pmatrix}$, $a\ge 0$ for  $x\ge 0$, and 
$\begin{pmatrix}1&a\\0&0\end{pmatrix}$, $a\ge 0$ for  $x\le 0$.
So $\overline {co} (f)$ is the set of all matrices
$\begin{pmatrix}1&a\\b&0\end{pmatrix}$, where $a\ge 0,\,0\le b \le 1 $.
Clearly $\Sigma \cap \overline{co} (f)$ are all the matricies such that 
$ab=0$. Hence $\Sigma \cap \overline{co} (f)$
does not contain any matrix from the relative interior of $\bar{co} (f)$. 
It can be easily checked that, for almost all $v\in \RR^2$, the set 
  $\Sigma_v(f)$ is just one matrix of the form 
$\begin{pmatrix}1&a\\0&0\end{pmatrix}$, $a>0$
which is  not an extremal point of  $\overline {co} (f)$. 
Hence the condition $(C_e)$ is not 
satisfied.
\end{exam}
\begin{exam}\label{Exampleantirelinta}
Similarly, let
$f:\RR^3\to\RR^3$ be an analytic function defined by
$f(x,y,z)= (y^3,zy^2,x+z^3)$.
Note that
 $f$ is nonsingular precisely
 outside 
$y=0$ and clearly is not constant on any segment. Moreover
$\overline {co} (f) $ has no singular matrices in its relative interior, and 
$f$ is obviously  not injective. 
$A(a)=\begin{pmatrix}0&0&0\\0&0&0\\1&0&a\end{pmatrix} \in 
\overline {co} (f), a \geq 0$, and 
we can use Proposition
\ref {extremalprop}, with $v=(1,0,-1) \in \ker (A(1/2)+A(3/2))$, to contradict 
the extremality 
condition $(C_e)$.

\end{exam}

Note that Theorem
\ref{THM12} holds for any continuous (possibly non-Lipschitz) function, 
however in Theorem \ref{THM2} we need   to assume that $f$ is 
locally Lipschitz. In fact we may drop this assumption but we need to 
strengthen the condition $(C_e)$. Precisely, we need to control the set
$\Sigma_v(f) $ at "infinity".
For any vector $v\in \RR^n$, $|v|=1$,  we put
$$
 \Sigma^c_v(f) : = \{A\in M(m,n) :\, Av=0\}\cap \overline {cone} (f).
$$
Clearly $\Sigma^c_v(f) $ is a closed convex  cone in the 
space
of matrices  $M(m,n)$. Now we can state our next theorem.

\begin{thm}\label{THM21}
Let $U$ be a convex open subset of $\RR^m$.
A  continuous tame mapping $f:U\to\RR^n$ is injective
if it satisfies the following conditions:
\begin{itemize}
\item[$(C^c_e)$] for any   $v\in \RR^n$, $|v|=1$ the set
$ \Sigma^c_v(f) $
is semi-extremal in $\overline {cone}(f)$,
\item[(S)] $f$ is not constant on any segment in $U$.
\end{itemize}
\end{thm}
\begin{rem} Actually   Theorem \ref{THM21} implies Theorem \ref{THM2}. 
To see this, note that for
a Lipschitz function, the set $\overline {co} (f)$ is compact and convex. 
It follows from Lemma \ref {generatingclosedcones} that $\overline {cone}(f)$ 
is a closed 
convex cone. Condition $(C^c_e)$ implies Condition $(C_e)$.
\end{rem}

We point out that the 
Example \ref{biLipshitzHommeoWithoutConvexRegularity}, a  biLipschitz 
homeomorphism 
$f:\RR^2\to\RR^2$,   does not satisfy the condition $(C_e)$.
This motivates  another generalization 
of the inverse mapping theorem. 
\begin{thm}\label{THM3}
Let $f:(\RR^n,0)\to(\RR^n,0)$ be a tame continuous map-germ.
Then $f$ is a homeomorphism, 
if there are  coordinates  systems 
$(x_1,\dots,x_n)$, $(y_1,\dots,y_n)$ of the source and  the target 
satisfying the following conditions:
\begin{itemize}
\item[($R_j$)] 
for  each $j=1,\dots,n-1$,
there are positive constants $K_j$, $L_j$ such that 
$$
K_j\le
\det\pd{(f_1,\dots,f_j)}{(x_1,\dots,x_j)}
\le L_j
\qquad
\textrm{except on $B(f)$}; 
$$ 
\item[($R_n$)] 
there is a positive constant $K_n$ such that 
$$
K_n\le\det\pd{(f_1,\dots,f_n)}{(x_1,\dots,x_n)}\qquad
\textrm{ except on $B(f)$.} 
$$
\end{itemize}
\end{thm} 
\begin{rem} Note that  a global version of this result is wrong. An analytic counter-exemple was 
given by Gale and Nikaido \cite{GaleNikaido}, when answering  a question of Samuelson whether the
strict positivity of upper-left principal minors  is a sufficient condition for the univalence.  
\end{rem}
We remark that $f$ and $f^{-1}$ may not be Lipschitz, as Example
\ref{NonLip} shows. 
We also remark that 
Condition ($R_n$) itself is not sufficient to ensure  the injectivity, 
as Example \ref{3sheet} shows. 
This theorem applies to Example
\ref{biLipshitzHommeoWithoutConvexRegularity},
but not to Example \ref{ExampleCubic}. 
So different generalizations are desirable. 
\begin{thm}\label{THM4}
A continuous tame map-germ $f:(\RR^n,0)\to(\RR^n,0)$ 
is a homeomorphism, 
if there are systems of coordinates 
$(x_1,\dots,x_n)$, $(y_1,\dots,y_n)$ of the source and the target 
which satisfy the following conditions:
\begin{itemize}
\item[($F_j$)]  for each  $j=1,\dots,n$
the mapping $\phi_j:(\RR^n,0)\to(\RR^n,0)$ defined by
$$
x=(x_1,\dots,x_n)\mapsto(f_1(x),\dots,f_j(x),x_{j+1},\dots,x_n)
$$ 
is finite; 
\item[($P_j$)] for each  $j=1,\dots,n$ we have
$0\le\det\pd{(f_1,\dots,f_j)}{(x_1,\dots,x_j)}$ except on $B(f)$. 
\end{itemize}
\end{thm} 
We remark that Conditions ($F_j$) cannot be dropped in Theorem \ref{THM4}, 
as we see in Example \ref{NonProper}.

\begin{exam}\label{NonProper}
Consider the mapping $f:(\RR^2,0)\to(\RR^2,0)$ defined by 
$$(x,y)\mapsto (x,x^2y).$$ 
Then the Jacobi's matrix of $f$ is
$$
df(x,y)
=
\begin{pmatrix}
1&0\\
2xy&y^2 
\end{pmatrix}.
$$  
The Condition $(C_e)$ is satisfied.
Indeed, observe that $\overline {co} (f)$ is the set of all matrices
$\begin{pmatrix}1&0\\a&b\end{pmatrix}$, where $a\in \RR$, $b\ge 0$.
For  $v=(0,\pm1)$ the corresponding set $\Sigma_v(f)$ is
the line in the boundary of $\overline {co} (f) $, so it is extremal
in $\overline {co} (f)$. For any other $v$  the set $\Sigma_v(f)$ is  empty
so is extremal in $\overline {co} (f)$. 
However, Condition (S) is not satisfied, 
since the image of $x$-axis is just the origin.
Obviously the map $f$ is  not a homeomorphism. 

Note that Condition ($R_1$) is satisfied but not Condition ($R_2$). 
Also note that Conditions ($F_1$), ($P_1$), ($P_2$) are satisfied, 
but Condition ($F_2$) is not satisfied. 
\end{exam}
\begin{exam}
Let $U$ be a convex open neighbourhood of $0$ in $\RR^2$. 
Consider the mapping $f:U\to\RR^2$ defined by 
$$
(x,y)\mapsto (x(y-x^2)^2,y-x^2).
$$ 
The Jacobi matrix of $f$ is
$$
df(x,y)
=
\begin{pmatrix}
a_{11}&a_{12}\\
a_{21}&a_{22}
\end{pmatrix}
=
\begin{pmatrix}
(x^2-y)(5x^2-y)&2x(y-x^2)\\
-2x&1
\end{pmatrix}. 
$$  
We show that $\RelInt({co}(f))$ contains singular matrices. 
Since $a_{11}=(x^2-y)(5x^2-y)$ changes its sign, it is enough to show
that the relative interior of ${co}(f)\cap\{a_{11}=0\}$ contains singular
matrices,  matrices with  $a_{12}a_{21}=0$. 
Since 
$$
\cD(f)\cap\{a_{11}=0\}
=\biggl\{
\begin{pmatrix}
0&0\\
-2x&1
\end{pmatrix}:(x,x^2)\in U
\biggr\}
\cup\biggl\{
\begin{pmatrix}
0&8x^3\\
-2x&1
\end{pmatrix}
:(x,5x^2)\in U\biggr\}, 
$$
 the relative interior of its convex hull contains singular matrices. 
This is clearly not a homeomorphism, 
the image of $\{y=x^2\}$ is just the origin. 
\end{exam}
\begin{exam}
Consider the mapping $f:\RR^2\to\RR^2$ defined by 
$$
(x,y)\mapsto(y_1,y_2)=(xy^{2/3},y^{1/3}).
$$
Its Jacobi's matrix is
$$
\begin{pmatrix}
y^{2/3} & \frac23xy^{-1/3}\\
0& \frac13y^{-2/3}
\end{pmatrix}
$$
and we have 
$$
\pd{y_1}{x}=y^{2/3}\ge0, \quad \pd{(y_1,y_2)}{(x,y)}=\frac13>0.
$$ 
Note the image of $x$-axis is the origin hence $f$ is not
injective.
We remark that this example satisfies 
Conditions $(C_e)$, ($R_2$), ($P_1$), ($P_2$), 
but not Conditions (S), ($R_1$) and  ($F_2$). 
\end{exam}
\begin{exam}
Consider the mapping $f:\RR^2\to\RR^2$ defined by 
$$
(x,y)\mapsto(y_1,y_2)=(xy^{1/3},y^{2/3}).
$$
The Jacobi's matrix is
$$
\begin{pmatrix}
y^{1/3}&\frac13xy^{-2/3}\\
0&\frac23y^{-1/3}&
\end{pmatrix}
$$
and its Jacobian determinant is $2/3$. 
Note the image of $f$ is $\{y_2>0\}\cup\{(0,0)\}$, which is not dense.  
We remark that this example satisfies 
Conditions $(C_e)$, ($R_2$), ($P_1$), ($P_2$), 
but not Conditions (S) and ($R_1$), ($F_2$). 
\end{exam}

Let $P:\RR^n-\{0\}\to (0,\infty)$ denote a  function satisfying
$$
P(t^{w_1}x_1,\dots,t^{w_n}x_n)=t^dP(x_1,\dots,x_n),\qquad t\in (0,\infty)
$$
where $w_1,\dots,w_n,d$ are real numbers. Let $F_P:\RR^n-\{0\}\to\RR^n-\{0\}$ be
 the mapping defined by 
$$
F_P(x_1,\dots,x_n)=
(P(x)^{w_1}x_1,\dots,P(x)^{w_n}x_n).
$$
If $(d+1)(d'+1)=1$ define $Q(x)$  by $Q(x)P^{(d'+1)}(x)\equiv 1$ and accordingly
 the corresponding map $F_Q$.

\begin{prop}\label{PropP}
$F_P^{-1}=F_Q$,  in particular $F_P$ is a homeomorphism whenever $P$ is continuous. 
If moreover   $(d+1)>0, \, w_i >0, \, i=1,\dots, n$, 
then  $F_P$ extends to a global homeomorphism of $\RR^n$.
\end{prop}
\begin{proof}
This follows from the following computation: 
\begin{align*}
F_Q\comp F_P(x)=
&F_Q(P(x)^{w_1}x_1,\dots,P(x)^{w_n}x_n)\\
=&\Bigl(
Q\bigl(P(x)^{w_1}x_1,\dots,P(x)^{w_n}x_n)\bigr)^{w_i}(P(x)^{w_i} x_i)
\Bigr)_{i=1,\dots,n}\\
=&\Bigl(\bigl(
P(x)^{d'}Q(x)\bigr)^{w_i}P(x)^{w_i} x_i\Bigr)_{i=1,\dots,n}\\
=&\bigl(
P(x)^{(d'+1)w_i} Q(x)^{w_i}x_i\bigr)_{i=1,\dots,n}\\
=&\bigl(x_i\bigr)_{i=1,\dots,n}.
\end{align*}
Finally observe that $\lim_{x\to 0} F_P= 0$  if
$dw_i +w_i>0$ for all \, $i=1,\dots, n$.
\end{proof}

\begin{exam}\label{biLipshitzHommeoWithoutConvexRegularity}
Consider the mapping $f:(\RR^2,0)\to(\RR^2,0)$ defined by 
$$
f(x,y)=
\begin{cases}
(P x,P y)&(x,y)\ne(0,0)\\
(0,0)&(x,y)=(0,0)
\end{cases}
\quad\textrm{where}\quad
P=\frac{x^4-x^2y^2+2y^4}{x^4-x^2y^2+y^4}.
$$
By Proposition \ref{PropP}, we see that this is a homeomorphism. 
Remark that the inverse is given by 
$$
(x,y)\mapsto
\begin{cases}
(P^{-1} x,P^{-1} y)&(x,y)\ne(0,0),\\
(0,0)&(x,y)=(0,0).
\end{cases}
$$ 
The Jacobi's matrix of $f$ is 
{\scriptsize
$$
\begin{pmatrix}
a_{11}&a_{12}\\
a_{21}&a_{22}
\end{pmatrix}
=\frac{1}{(x^4-x^2y^2+y^4)^2}
\begin{pmatrix}
(x^2-2y^2)(x^6-y^6)&
2x^3y^3(2x^2-y^2)\\
2xy^5(y^2-2x^2)
&x^8-2x^6y^2+8x^4y^4-5x^2y^6+2y^8
\end{pmatrix},
$$}
and we observe each entry is bounded. Hence $f$ is Lipschitz. 
Similarly we can show $f^{-1}$ is Lipschitz,  
so $f$ is a biLipschitz homeomorphism. 
We observe that  
{\footnotesize
$$
\pd{f_2}{y}=\frac
{(x^2-y^2)^4+y^4(x^2-y^2)^2+x^2y^2(2x^4+x^2y^2+y^4)}
{(x^4-x^2y^2+y^4)^2}\ge1, 
\textrm{ and }
\det(df)=P^2\ge1.
$$} 
We  use Theorem \ref{THM3}  to show that $f$ 
is a homeomorphism germ. 
Consider the segment 
$$
\gamma_\varepsilon:
[-2\varepsilon,2\varepsilon]\to\RR^2,\qquad
t\mapsto(t,\pm\varepsilon)
\qquad\textrm{for $0<\varepsilon\ll1$}.
$$
The images of $f\comp\gamma_{\varepsilon}$ look like the following:
\vskip 1cm
\input{figureA}
\vskip 1cm
We now show that $\RelInt({co}(f))$ contains singular matrices, hence in 
particular
$f$ does not satisfy the condition $(C_e)$.
First one can see that the affine hull of ${co}(f)$ is $M(2,2)$.
Indeed the following matrices from ${co}(f)$ are affine independent 


$$
\begin{pmatrix}
1&0\\
0&1
\end{pmatrix},\quad
\begin{pmatrix}
0&-2\\
2&4
\end{pmatrix},\quad
\begin{pmatrix}
0&2\\
-2&4
\end{pmatrix},
\quad
\begin{pmatrix}
0&-\frac{2\sqrt{2}}{3}\\
\frac{4\sqrt{2}}3&\frac83
\end{pmatrix},\quad
\begin{pmatrix}
0&\frac{2\sqrt{2}}{3}\\
-\frac{4\sqrt{2}}3&\frac83
\end{pmatrix}
$$
Next we observe that 
$a_{11}=\frac{(x^2-2y^2)(x^6-y^6)}{(x^4-x^2y^2+y^4)^2}$ changes its
 sign. So it is enough to show that the relative interior of 
${co}(f)\cap\{a_{11}=0\}$ contains a singular matrix. 
If $a_{11}=0$, then $y=\pm x, \pm x/\sqrt{2}$, and we observe
that the last four matrices from the list above
are realized as Jacobi's matrices precisely in $\{a_{11}=0\}$.
Since the singular matrices in the space $\{a_{11}=0\}$
satisfy $\{a_{12}a_{21}=0\}$, 
the relative interior of the convex hull of those four matrices 
contains singular matrices. This implies that
$\RelInt({co}(f))$ contains singular matrices.
\end{exam}
This example shows that Condition $(C_e)$
is not invariant by biLipschitz mappings.  
\begin{exam}\label{NonLip}
Let $f:(\RR^2,0)\to(\RR^2,0)$ be a mapping defined by
$$
f(x,y)=
\begin{cases}
(P^3 x,P^2 y)&(x,y)\ne(0,0)\\
(0,0)&(x,y)=(0,0)
\end{cases}
\quad\textrm{where}\quad
P=\frac{2x^4+y^6}{x^4+y^6}
$$
By Proposition \ref{PropP}, we see that this is a homeomorphism. 
We obtain that 
$$
df=
\frac{(2x^4+y^6)}{(x^4+y^6)^4}
\begin{pmatrix}
(2x^4+y^6)(2x^8+15x^4y^6+y^{12})&-18x^5y^5(2x^4+y^6)\\
8x^3y^7(x^4+y^6)&(x^4+y^6)(2x^8-9x^4y^6+y^{12})
\end{pmatrix}, 
$$
and observe that all components but $\pd{f_2}{x}$ are bounded,
so this mapping is not Lipschitz. 
Moreover, we observe that  
$$
\pd{f_1}{x}=\frac{(2x^4+y^6)(2x^8+15x^4y^6+y^{12})}{(x^4+y^6)}\ge1,
\textrm{ and }
\det(df)=P^5\ge1.
$$
So we can use Theorem \ref{THM3} also to show that $f$ 
is a homeomorphism germ. Consider the segment 
$$
\gamma_{x_0}:
[-\varepsilon,\varepsilon]\to\RR^2,\qquad
t\mapsto(x_0,t)
\qquad\textrm{for}\quad 0<\varepsilon\ll|x_0|\ll1. 
$$
Then we have  
$$
df(\gamma_{x_0}(t))\gamma_{x_0}'(t)=
\frac{(2x_0^4+t^6)}{(x_0^4+t^6)^4}
(-18x_0^5t^5(2x_0^4+t^6),(x_0^4+t^6)(2x_0^8-9x_0^4t^6+t^{12}))
$$
and we observe that the first component changes the sign at $t=0$ and 
the second component changes the sign four times.
Thus the images of $f\comp\gamma_{x_0}$ look like the following:

\vskip 1cm
\input{figureB}
\vskip 1cm
We see that $\RelInt({co}(f))$ contains singular matrices.
\end{exam}
\begin{exam}\label{3sheet}
Let $f:(\RR^2,0)\to(\RR^2,0)$ be a mapping defined by
$$
f(x,y)=
\begin{cases}\frac{1}
{x^2+y^2}\bigl(x(x^2-3y^2),y(3x^2-y^2)\bigr)&(x,y)\ne(0,0)\\
(0,0)&(x,y)=(0,0)
\end{cases}
$$
Then we have 
$$
df=\frac{1}{(x^2+y^2)^2}
\begin{pmatrix}
x^4+6x^2y^2-3y^4&8xy^3\\
-8x^3y&3x^4-6x^2y^2-y^4
\end{pmatrix}
$$
and $\det (df)=3$. 
We then obtain that 
\begin{align*}
&\frac13df(r,0)+\frac13df(r\cos\tfrac{2\pi}{3},r\sin\tfrac{2\pi}3)
+\frac13df(r\cos\tfrac{2\pi}{3},-r\sin\tfrac{2\pi}3)
\\
&=
\frac13
\begin{pmatrix}
1&0\\
0&3
\end{pmatrix}
+\frac13
\begin{pmatrix}
-\frac12&\frac{\sqrt{3}}{2}\\
-\frac{3\sqrt{3}}{2}&-\frac32
\end{pmatrix}
+\frac13
\begin{pmatrix}
-\frac12&-\frac{\sqrt{3}}{2}\\
\frac{3\sqrt{3}}{2}&-\frac32
\end{pmatrix}
=\begin{pmatrix}
0&0\\
0&0
\end{pmatrix}, 
\end{align*}
and so $\RelInt({co}(f))$ contains a singular matrix. 
We see that 
$$
f(r\cos\theta,r\sin\theta)=(r\cos3\theta,r\sin3\theta) 
$$
and that the mapping $f$ is a 3-sheeted branched covering at $0$.  
\end{exam}

\lsection{O-minimal key lemma}\label{S:KeyLemma}
The purpose of this section is to show the key Lemma 
\ref{o-minKeyLemma}. It will allow us to extend our control on directional 
derivatives 
from an open dense  set to the whole domain of the considered map.

\begin{defn}
An {\bf o-minimal structure} on $(\RR,+, \cdot, <)$  
(cf. \cite{Coste} or \cite{vandebDries}) is a sequence of
boolean algebras $\cO_n$ 
of {\bf definable} subsets of $\RR^n$, 
such that for each $n\in\NN$
\begin{itemize}
\item
if $A\in\cO_m$ and $B\in\cO_n$, 
then $A\times B\in\cO_{m+n}$;
\item
if  $\Pi:\RR^{n+1}\to \RR^n$ is the canonical projection onto 
$\RR^n$ then for any $A\in\cO_{n+1}$, the set $\Pi(A)$ belongs to $\cO_n$;
\item
$\cO_n$ contains the family of algebraic subsets of $\RR^n$, 
that is, every set of the form $\{x\in\RR^n : p(x) = 0\}$,
where $p:\RR^n\to\RR$ is a polynomial function;
\item
the elements of $\cO_1$ are exactly the finite unions of intervals and points.
\end{itemize}
\end{defn}
\begin{defn}
Given an o-minimal structure  $\cO$ on $(\RR,+, \cdot,<)$ and   $U\subset \RR^m$, we say that a mapping 
$U\to\RR^n\cup\{+\infty\}$ is 
{\bf definable} in $\cO$ 
if its graph belongs to $\cO_{m+n}$. We say for short  that $f$ is definable if it is definable in some
o-minimal structure $\cO$ on $(\RR,+, \cdot,<)$.
\end{defn}

Let $U$ be an  open set of $\RR^m$ and 
let $f:U\to\RR^n$ be a continuous definable mapping. 
Recall that  the set ${B}(f)$ of the  points where $f$ is 
 not differentiable, is a nowhere dense subset  of $\RR^m$.
If $A=df(x)\in \cD(f)$, then $Av=\partial_v f(x) \in \cD_v(f)$, the 
directional derivative   of $f$ at $x$ in  the direction $v$.

We state now the key  lemma.
\begin{lem}\label{o-minKeyLemma}
Let $U$ be an  open subset of $\RR^m$ and 
let $f:U\to\RR^n$ be a continuous definable mapping. 
Fix $x\in U$ and $v\in\RR^m$, $|v|=1$.
Let  $g(t)=f(x+tv)$, $t\in[a,b]$.
Then $g'(t)\in\overline{\cD_v(f)}$, whenever
$g'(t)$ exists.

\end{lem}
\begin{proof}
Clearly it is enough to  consider the case   $n=1$.
Since $f$ is a continuous definable mapping, $g$ is a continuous 
definable function. Hence,  see  \cite{Coste} or \cite{vandebDries},  
$g'(t)$ exists (and is continuous) except for finitely many  points.
 It is thus enough to show the lemma for a generic $t$, that is, for all but 
finitely many $t\in[a,b]$.
So, in the  proof we shall replace the segment $[a,b]$ by a  subsegment,  
but for simplicity we shall 
denote  it again by $[a,b]$.
  First we explain the reduction to the case $m=2$.
Let  $S_v$ denote the unit sphere in the orthogonal complement of $v$. For 
any $u\in S_v$
we define a map 
$$ 
p_{u,v}:[a,b]\times[0,+\infty )\ni (t,s)\mapsto x+tv+su\in \RR^m.
$$

\begin{lem}\label{2dimreduction}
There exists $u\in S_v$ such that 
$\dim p_{u,v}^{-1} (B(f)) = 1.$ Actually the set of all such $u$ is dense in $S_v$.
\end{lem}
\begin{proof}  Assume that there exists  an open non-empty  set  
$ W\subset S_v$
 such that the corresponding sets  $p_{u,v}^{-1} (B(f))$, $u\in W$,
are of dimension $2$. We have a natural projection 
$$\pi:\RR^m\setminus \RR v \to S_v,$$ which
is a   composition of: the translation by $-x$, the orthogonal projection on the orthogonal complement 
of $v$ and next the radial projection on $S_v$.
Note that for any $u \in S_v$, the   image of   $p_{u,v}$
is contained in the fiber $\pi^{-1}(u)$.
It follows from the formula for dimension  of a definable set that
$$\dim B(f) \ge min\{ \dim \pi^{-1}(u)\cap B(f),\, u\in W\} + \dim W =2+(m-2)=m,$$
which is a contradiction, since  $B(f)$ is a nowhere dense set, so 
$\dim B(f) <m$. 
\end{proof}
To achieve the proof of Lemma \ref{o-minKeyLemma} we fix $u\in S_v$ such that 
$\dim p_{u,v}^{-1} (B(f)) = 1$.
Then, for some $\varepsilon >0$ small enough, we consider the following  
continuous definable function
$$F(t,s)=f(x_0+tv+su)-f(x_0+tv), \, t \in[a,b], \, s\in [0,\varepsilon).$$
Clearly $F(t,0)=0$. Since $\dim p_{u,v}^{-1} (B(f)) = 1$, applying the cell 
decomposition (c.f. \cite{Coste} or \cite{vandebDries}) to $p_{u,v}^{-1} (B(f))$
we may assume (changing suitably $a$ and $b$) that  $F$ is differentiable 
(even $2$-times
differentiable)  in
$[a,b]\times(0,\varepsilon)$. Moreover  
$\pd{F}{t}(t,s) = \partial_v f (p_{u,v}(t,s)) -g'(t)$.
So it is enough to show $\displaystyle{\lim_{s\to0}\pd{F}{t}(t,s)}=0$, for
 generic $t$.  Consider the definable set 
$$
Z=\biggl\{(t,s)\in(t^*-\delta,t^*+\delta)\times(0,\varepsilon):
\pd{F}{t}(t,s)=0,\ \frac{\partial^2F}{\partial t^2}(t,s)=0\biggr\}.
$$
Applying  the cell decomposition to $Z$ we may assume that
$\pd{F}{t}(t,s)$ and $\frac{\partial^2F}{\partial t^2}(t,s)$ are of  
 constant sign on $[a,b]\times(0,\varepsilon)$.
Let us assume that both partials are strictly positive (the other cases 
are similar). 

We claim that, for generic $t$,  
$\pd{F}{t}(t,s)\to0$ as $s\to0$. 
 Assume that for some 
$t^*\in [a,b)$ we have  
 $\pd{F}{t}(t^*,s)\to 2c>0$, as $s\to0$. (A  priori we may have $c=+\infty$,  
we leave to the reader to 
 adapt the argument below to this case.)
 For each $s\in (0,\varepsilon)$ the function $t\to\pd{F}{t}(t,s)$ 
is increasing, so for any $s>0$ small enough we have 
$$
\pd{F}{t}(t,s)\ge c,\, t\in[t^*,b]. 
$$ 
By the Mean Value Theorem it
 follows that 
$$
{F}(t,s)\ge  F(t^*,s) +c(t-t^*), t\in[t^*,b]. 
$$
 So taking limit as  $s\to0$, we obtain
 $$
{F}(t,0)\ge  F(t^*,0) +c(t-t^*), t\in[t^*,b]. 
$$
But this is absurd since  ${F}(t,0)=0, t\in[t^*,b]$. 
\end{proof}

\subsection{Tame, subanalytic and definable mappings}

We explain below the notion of tame mapping we use in this paper.
\begin{defn}
$X\subset\RR^n$ is {\bf semianalytic} if, 
for all $x\in\RR^n$, there is an open
neighborhood $U$ of $x$ such that $X\cap U$ is a finite Boolean
combination of  sets $\{x\in U:f(x)=0\}$ and $\{x\in U:g(x)>0\}$, 
where $f,g:U\to\RR$ are analytic.
\end{defn}
\begin{defn}
$X\subset\RR^n$ is {\bf subanalytic} (cf. \cite{Gabrielov}, 
\cite{Hironaka})
if, for all $x\in\RR^n$, there is an open set
$U$ and a bounded semianalytic set $Y\subset\RR^{n+m}$  
such that $X\cap U$ is the projection of $Y$ into $U$.
\end{defn}
 Let us recall that  the  collection of global subanalytic sets  
form an o-minimal structure.
(Recall that $A\subset \RR^n$ is {\bf globally subanalytic}  
if  $A$ is subanalytic as subset of the projective space $\PP^n$, for  the 
natural embedding  $\RR^n \to\PP^n$). In particular, if $U\subset \RR^n$ is 
bounded and  $f: U\to \RR$  is a bounded subanalytic function then $f$ is 
definable, that is, $f$ is globally subanalytic.  
Both conditions are actually necessary. 

Recall that we call  a continuous map $f:U\to \RR^n$, where $U\subset \RR^m$ is open, 
{\bf tame} if its graph is definable in an o-minimal structure or subanalytic. 
Note that  Lemma   
\ref {o-minKeyLemma} applies also to subanalytic functions since we may assume that
$f$ is bounded.

\subsection{Directional derivatives of tame mappings }\label{dirder}
Let $v\in S^{m-1}$. 
Recall that $$
\partial_v f(x)=\lim_{h\to+0}\frac{f(x+hv)-f(x)}{h}.
$$
is the directional derivative   
of $f$ at $x$ in  direction $v$.
Let  $f:U\to \RR^n$ be a tame function, where $U$ is an open subset of $\RR^n$.

Let $B_v(f)$ be the set of points $x\in U$ such that $\partial_v f(x)$ does 
not exists
(precisely equals $+\infty$ or $-\infty$).  Recall that  
$B_v(f) \subset B(f)$, where
$B(f)$ is the set of points at which $f$ is not differentiable.
We denote  
$$Dir_v (f) =\{\partial_v f(x): \, x\in U\setminus B_v(f)\},$$
and 
$$
\overline {co}_v  (f) =\overline {co}(Dir_v(f)).$$

Recall that $\cD_v(f)=\{Av:A\in\cD(f)\}$, where 
$\cD(f)=\{df(x):x\in U-B(f)\}$. 
Note that   Lemma \ref {o-minKeyLemma} 
implies the following statement.
\begin{prop}\label{Prop:dirder}
Let $U$ be an open subset  of $\RR^n$ and $f: U \to \RR^n$  a  continuous 
mapping which is  tame
(i.e subanalytic or definable in an o-minimal structure). Then 
$\overline{ Dir_v (f) } =  \overline{ \cD_v (f)}$,
hence
$$\overline {co}_v(f)=\overline {co}(\cD_v(f)).$$
\end{prop}

Let us  consider the linear map $\varphi: M(m,n)\to \RR^n$,
$\varphi (A) = Av$. We  know by Proposition \ref{Prop:dirder}
that  $\overline{co}_v(f)=\overline {co}(\varphi ({co}(f))$.
Let us assume  now that $f$ is Lipschitz, 
hence $\overline{co}(f)$ is compact. Thus
$$\overline {co}_v(f)=\overline {co}(\varphi ( \overline {co}(f)) 
=\varphi ( \overline {co}(f)).
$$
This shows that the  condition $(C_e)$ and   Lemma \ref{extremalimage} imply
the following important fact.
\begin{lem}\label{crucialemLip} If $f$ is Lipschitz and satisfies condition 
$(C_e)$,  then 
  the vector $0\in \RR^n$ is an extremal point in $\overline {co}_v(f)$.
\end{lem}

Similarly, for the need of the proof of Theorem \ref{THM21}, we consider 
closed 
convex cones.
We put $\overline{cone}_v(f):= \overline{cone}( \cD_v(f))$. Then
 Proposition \ref{Prop:dirder}  and Lemma  \ref{extremalconimage} yields
$$\overline{cone}_v(f)=\varphi (\overline{cone}(f)),$$
since $\varphi (\overline{cone}(f))$ is a closed convex cone.

Note that  $\overline {co}_v(f) \subset \overline{cone}_v(f)$, so we conclude 
from the above discussion that the condition $(C^c_e)$ and  
Lemma \ref{extremalimage} imply
the following important fact.
\begin{lem}\label{crucialemnonLip} If $f$ is continous and satisfies 
condition $(C^c_e)$,  then 
  the vector $0\in \RR^n$ is an extremal point in $\overline {cone}_v(f)$. 
Hence, also  $0\in \RR^n$ is an extremal point in $\overline {co}_v(f)$.
\end{lem}

\lsection{Proof of Theorems \ref{THM1}, \ref{THM12}, \ref{THM2}.}
\label{S:Proof12}

\begin{lem}\label{LemmaK1}
Take $A\in M(m,n)$ and $v\in S^{m-1}$. 
If $\dist(A,\Sigma)\ge\delta$, then $|Av|\ge\delta$. 
\end{lem}
\begin{proof}
By Lemma \ref{LemmaKOS}, we have 
$\dist(A,\Sigma)=\inf\{|Av|:v\in S^{m-1}\}$, so the lemma is clear. 
\end{proof}
Next we state a classical lemma on the  projection on a convex closed subset.
\begin{lem}\label{LemmaK2}
Let $C\subset \RR^n$ be a convex and closed subset with $0\not\in C$. 
Then there exists unique $w\in C$ such that
\begin{itemize}
\item $|w|=\inf\{|u|:u\in C\}$.
\item if $x\in C$ then $\langle\frac{w}{|w|},x\rangle\ge|w|>0$.
\end{itemize}
\end{lem}
\begin{proof}[Proof of Theorem \ref{THM1}]
Take $x,x'\in U$,  $x\ne x'$ and put
$v=\frac{x'-x}{|x'-x|} $, 
 $v\in S^{m-1}$. We consider the set 
$$
{co}_v(f)=\{Av:A\in {co}(f)\}.
$$ 
By  Lemma \ref{LemmaK1}, if $u\in \overline{{co}_v(f)}$ then $|u|\ge\delta$. 
So Lemma \ref{LemmaK2} implies that there exists a unique $w\in \overline{{co}_v(f)}$ 
such that 
$$
|w|=\inf\{|x|:x\in{co}_v(f)\}.
$$
Note that by Lemma \ref{LemmaK1}, we have $|w|\ge \delta$. 
Hence  Lemma \ref{LemmaK2} yields 
$$
\Bigl\langle\frac{w}{|w|},u\Bigr\rangle\ge\delta
\quad\textrm{ if $u\in {co}_v(f)$.}  
$$ 
Finally consider the function
$$
h(t)=
\Bigl\langle\frac{w}{|w|},f(x+tv)-f(x)\Bigr\rangle,\qquad
t\in[0,|x'-x|]. 
$$
Moving slightly $x$ and $x'$, we may assume that, 
for almost all (in fact, all but finite, since $f$ is tame) 
$t\in[0,|x'-x|]$, the function $f$ is differentiable at $x+tv$. 
Since $h'(t)=\langle\frac{w}{|w|},df(x+tv)v\rangle\ge\delta$, we obtain that
$$
h(|x'-x|)=\int_0^{|x'-x|}h'(t)dt\ge\delta|x'-x|. 
$$
So we have proved that $|f(x')-f(x)|\ge h(|x'-x|)\ge\delta|x'-x|$.  
Note that, if $m=n$ then
$f:U \to f(U)$ is open by the invariance of domain 
(Lemma \ref{InvarianceOfDomain}), 
so indeed $f:U \to f(U)$ is a homeomorphism. 
\end{proof}

We recall  now the classical result of Minkowski about supporting hyperplanes,
see e.g. \cite {HiriantUrruty-Lemarechal}.
\begin{lem}\label{LemmaMinkowski}
Let $C\subset \RR^n$ be a convex set, and assume that 
relative interior of $C $ does not contain $0$.
Then there exists  $w\in \RR^n$, $|w|=1$ such that
 $\langle{w},x\rangle \ge0$, for all $x\in C$.
\end{lem}

For the proof of the next two theorems the following lemma is crucial.
\begin{lem}\label{KeyLem}
Let $g:[a,b]\to\RR^n$, $a<b$, be a tame continuous  function.
If the relative interior of ${co}(g)$ does not contain $0$,
then there is $w\in\RR^n$ such that
$\langle w,g(a)\rangle<\langle w,g(b)\rangle$.
This obviously implies that $g(a)\ne g(b)$.
\end{lem}
\begin{proof}
We show this lemma by induction on $n$.
Assume $n=1$. Since the relative interior of ${co}(g) $ does not contain $0$,
it follows  that either  $g'(t)\ge0$ for all but finitely many $t\in [a,b]$ 
or $g'(t)\le0$ for all but finitely many $t\in [a,b]$. 
Hence $g$ is  
monotonic. 
Note that $g$ is non constant, since otherwise  ${co}(g)=\{0\}$.
Therefore $g(a)\ne g(b)$.

Assume now  that $n>1$.
Since $0$ is not in the relative interior of ${co}(g)$, by Lemma \ref{LemmaMinkowski},
 there exists
$w\in\RR^n$, $|w| =1$ such that $\langle w,g'(t)\rangle\ge0$.
If $\langle w,g'(t)\rangle=0$, for all but finitely many $t\in [a,b]$ , then 
$g(t)\in H$, $t\in [a,b] $,
 for some  affine hyperplane $H$ 
 orthogonal to $w$. Hence we can apply the induction hypothesis.
Otherwise the affine hull of ${co}(g)$ is $\RR^n$, and
$\langle w,g'(t)\rangle$ is  strictly positive on an open interval.
This implies 
$$
\langle w,g(b)\rangle-\langle w,g(a)\rangle
=\int_{a}^{b}\langle w,g'(t)\rangle dt>0.
$$
\end{proof}

\begin{proof}[Proof of Theorem \ref{THM12}]
First let us observe that by the classical local inverse theorem and by the 
invariance of domain, we deduce that $f$  is open on $U\setminus B$.
Let us fix  $x,x'\in U$,  $x\ne x'$. Of course, if $x,x'\in B$ then, by the 
condition $(I)$, we have
$f(x)\ne f(x')$.

 Let us assume now that $x\in U\setminus B$ and that $f(x)=f(x')$. 
Since $f$ is open at $x$, there exists $V$ an open neighborhood of $x$  such 
that $f(V)$ is an
open  neighborhood of $f(x)$. By the continuity of $f$ at $x'$ there exists 
$V'$ an open neighborhood of $x'$  such that $V\cap V'=\emptyset$, 
$f(V')\subset f(V)$.
Recall that $B$ is closed and nowhere dense in $U$ hence also in $V$. So 
$W' = V'\setminus B \ne \emptyset$ is open. Now let $W= V\cap  f^{-1}(f(W'))$.
Note that  for any $x'\in W'$ there exists $x \in W$, hence 
$x\ne x'$,  such  that
$f(x)=f(x')$.
 So we can actually assume that both 
$x, x'\in U\setminus B$.
Now applying Lemma \ref{omiting} we can also suppose that
$[x,x']\cap B$ is finite.

%

Let us  put
$v=\frac{x'-x}{|x'-x|} $. Set $g(t)=f(x+tv)$,  $t\in [0,|x'-x|] $.
Let us  consider the linear map $\varphi: M(m,n)\to \RR^n$,
$\varphi (A) = Av$.  Then 
${co}_v(f):=\varphi({co} (f,U\setminus B))$ is a convex subset of $\RR^n$ 
which does not contain $0$.
Note that
$$
g'(t)= \partial_v f(x+tv)
$$
 at any point $t\in[0,|x'-x|] \setminus F$, where $F$ is a finite set. So we 
have
$$
{co}(g) \subseteq  {co}_v(f).
$$
We know that $0\notin {co}_v(f)$, hence $0\notin {co}(g)$, where ${co}(g)$ is 
precisely
the convex hull $\{g'(t): \,t\in [a,b]\setminus F\}$.
Thus, applying Lemma \ref{KeyLem} we obtain
 Theorem \ref{THM12}.
 \end{proof}

\begin{proof}[Proof of Theorem \ref{THM2}]
Let us fix  $x,x'\in U$,  $x\ne x'$ and put
$v=\frac{x'-x}{|x'-x|} $. Set $g(t)=f(x+tv)$,  $t\in [0,|x'-x|] $.
Note that
$$
g'(t)= \partial_v f(x+tv)
$$
 at any point $t\in[0,|x'-x|] $ such that $g$ is differentiable at $t$. 
So we have
$$
{co}(g) \subseteq \overline {co}_v(f).
$$

Recall that in Theorem \ref{THM2} the function $f$ is supposed to be  locally Lipschitz.
Since any open convex subset of $\RR^n$ is a union of increasing family
of compact convex sets we may assume that $f$ is actually
 Lipschitz. Hence, by Lemma \ref{crucialemLip}   we know that $0$ is an extremal
point of $\overline {co}_v(f)$. It follows from Lemma \ref{extremalsubset} that
$0$ is also an extremal
point of ${co}(g)$. But  ${co}(g)\ne \{0\}$ since otherwise $f$ is constant on $[x,x']$.
Thus  $0$ is not in the relative interior of  ${co}(g)$, so it is enough to apply Lemma \ref{KeyLem} to
obtain Theorem \ref{THM2}.

\end{proof}

\begin{proof}[Proof of Theorem \ref{THM21}]
It goes along the same line as the above proof of Theorem \ref{THM12},
however to conclude we use  Lemma \ref{crucialemnonLip}.
\end{proof}

\lsection{Proof of Theorems \ref{THM3}, \ref{THM4}.}\label{S:Proof34}
Throughout this section, $f:(\RR^n,0)\to(\RR^n,0)$ denotes a continuous tame map-germ. 
Let $y_1,\dots,y_n$ denote a coordinate system of the target. Set  
$f_j=y_j\comp f$. We shall not distinguish between a mapping and its germ.
Consider the mapping $\phi_k:(\RR^n,0)\to(\RR^n,0)$, defined by 
$$
(x_1,\dots,x_n)\mapsto(f_1(x),\dots,f_k(x),x_{k+1},\dots,x_n). 
$$ 
We define $x^{(k)}=(x_1^{(k)},\dots,x^{(k)}_n)$, $k=1,\dots,n,n+1$ 
by the following relations: 
$$
x_j=x_j^{(1)}=\dots=x_j^{(j)},\qquad
x_j^{(j+1)}=\dots=x_j^{(n+1)}=f_j(x), \qquad j=1,\dots,n. 
$$
Then the mapping $\phi_k$ is expressed by $x\mapsto x^{(k+1)}$. 
If $\phi_{k-1}$ is a homeomorphism, 
$x^{(k+1)}$ is considered as a mapping of $x^{(k)}$. 
In this case, we have the following 
\begin{lem}\label{LemJac} 
$\pd{x^{(k+1)}_k}{x^{(k)}_k}
=
\det\pd{(f_1,\dots,f_k)}{(x_1,\dots,x_k)}
\bigg/
\det\pd{(f_1,\dots,f_{k-1})}{(x_1,\dots,x_{k-1})}$
\end{lem}
\begin{proof}
Chain rule.
\end{proof}

\begin{lem}\label{Cor1}.
Consider a continuous tame mapping 
$$
f:(\RR^n,0)\to(\RR^n,0)\quad
\textrm{defined by}\quad
(x_1,\dots,x_n)\mapsto(f_1(x),x_2,\dots,x_n).
$$
If there is a positive constant $K$ such that 
$K\le\pd{f_1}{x_1}$, then 
$f$ is a homeomorphism.
\end{lem}
\begin{proof}
It is enough to show that $f$ is injective. 
Let $x\ne x'$ be two points in $U$. We set $v=\frac{x'-x}{|x'-x|}$.
If $v\ne(1,0,\dots,0)$, then it is clear that $f(x)\ne f(x')$.  
Assume that $v=(1,0,\dots,0)$. 
By  Lemma \ref{o-minKeyLemma}), 
we conclude that $\pd{f_1}{x_1}$ is positive, 
and $f_1(x)\ne f_1(x')$. This completes the proof.
\end{proof}
\begin{proof}[Proof of Theorem \ref{THM3}]
By the induction on $k$ we show that each $\phi_k$, $k=1,\dots,n$ is a homeomorphism. 
By Lemma \ref{Cor1}, we obtain $\phi_1$ is a homeomorphism.
Assume that $\phi_k$ is a homeomorphism. Then, by Lemma \ref{LemJac}, 
we have that 
$$ 
\pd{x^{(k+2)}_{k+1}}{x^{(k+1)}_{k+1}}
=
\det\pd{(f_1,\dots,f_{k+1})}{(x_1,\dots,x_{k+1})}
\bigg/
\det\pd{(f_1,\dots,f_{k})}{(x_1,\dots,x_{k})}
\ge \frac{K_{k+1}}{L_k}. 
$$
Applying Lemma \ref{Cor1}, we obtain that $\phi_{k+1}\comp\phi_k^{-1}$ is a
 homeomorphism. 
Thus the mapping $\phi_{k+1}=(\phi_{k+1}\comp\phi_k^{-1})\comp\phi_k$ is a homeomorphism.
\end{proof}

For the proof of Theorem \ref{THM3} we shall need a following    variant of the previous lemma.
\begin{lem}\label{Cor2}
Consider a continuous finite tame mapping 
$$
f:(\RR^n,0)\to(\RR^n,0)\quad
\textrm{defined by}\quad
(x_1,\dots,x_n)\mapsto(f_1(x),x_2,\dots,x_n).
$$
If $0\le\pd{f_1}{x_1}$, then 
$f$ is a homeomorphism.
\end{lem}
\begin{proof}
It is enough to show that $f$ is injective. 
Let $x\ne x'$ be two points in $U$. We set $v=\frac{x'-x}{|x'-x|}$.
If $v\ne(1,0,\dots,0)$, then it is clear that $f(x)\ne f(x')$.  
Assume that $v=(1,0,\dots,0)$. 
By  Lemma \ref{o-minKeyLemma}, 
we conclude that $\pd{f_1}{x_1}$ is non-negative. 
Since $f$ is finite, $t\mapsto f_1(x+tv)$ is not constant. 
This completes the proof.
\end{proof} 
\begin{proof}[Proof of Theorem \ref{THM4}] 
Again, by induction on $k$ we show that each $\phi_k$, $k=1,\dots,n$ is a homeomorphism. From Lemma \ref{Cor2}, we obtain that $\phi_1$ is a homeomorphism.
Assume that $\phi_k$ is a homeomorphism. Then, by Lemma \ref{LemJac}, 
we have that 
$$ 
\pd{x^{(k+2)}_{k+1}}{x^{(k+1)}_{k+1}}
=
\det\pd{(f_1,\dots,f_{k+1})}{(x_1,\dots,x_{k+1})}
\bigg/
\det\pd{(f_1,\dots,f_{k})}{(x_1,\dots,x_{k})}
\ge 0. 
$$
Applying Lemma \ref{Cor2}, we obtain that $\phi_{k+1}\comp\phi_k^{-1}$ is a
 homeomorphism. 
Thus the mapping $\phi_{k+1}=(\phi_{k+1}\comp\phi_k^{-1})\comp\phi_k$ is a homeomorphism.
\end{proof}

\lsection{Remark on homotopy type of $df$.}\label{S:Contract}
It looks interesting to investigate the homotopy type 
of the mapping
$$
df:(\RR^n-\widehat{B}(f),0)\to \GL(n,\RR),
$$
where $\widehat{B}(f)$ stands for the set of points at which $f$ is not differentiable 
or the differential exist but is singular.
By Cartan-Iwasawa  decomposition  (see e.g. \cite{Helgason}) we have a homeomorphism 
$$
\GL(n,\RR)\simeq\OT(n)\times\RR^{n(n+1)/2}.
$$
Let $\GL_+(n,\RR)$ denote the set of invertible matrices with positive
determinants. 
Then the above mapping induces a homeomorphism:
$$
\GL_+(n,\RR)\simeq\SO(n)\times\RR^{n(n+1)/2}.
$$
Let $p:\GL_+(n,\RR)\to\SO(n)$ denote the projection to the first
component. 
\begin{prop}\label{THM:Contract}
Let $f:(\RR^2,0)\to(\RR^2,0)$ be a continuous tame mapping with
$\widehat{B}(f)=\{0\}$ and $f^{-1}(0)=\{0\}$.
If the mapping 
$$
df:(\RR^2-0,0)\to\GL(2,\RR)
$$
is null homotopic, then $f$ is a homeomorphism.
\end{prop}
\begin{proof}
We may assume that $df$ is a mapping into $\GL_+(2,\RR)$. Then,  
by the assumption, the mapping $p\comp df$ is null homotopic. 
We remark that $\pi_1(\GL_+(2,\RR))=\pi_1(\SO(2))=\ZZ$. 
Since the homotopy type of the mapping  
$$
\phi_f:(\RR^2-0,0)\to S^1,\qquad 
(x,y)\mapsto 
\frac{f(x,y)}
{|f(x,y)|}
$$
determines the topological type of $f$, we can construct a regular 
homotopy $H_t$ ($t\in [0,1]$) between $H_1=f$ and $H_0$ where 
$H_0$ is the complex function $z^k$ for some integer $k$. 
This homotopy induces a homotopy $p\comp dH_t$. 
Since it is null homotopic, we conclude $k=1$. 
\end{proof}
The analogous statement is not true for $n\ge3$,    
since $\pi_1(\SO(n))=\ZZ/2\ZZ$ ($n\ge3$).  
\begin{exam}
Consider the mapping 
$$
f:\RR^3\to\RR^3\quad \textrm{defined by}\quad(x,y,z)\mapsto
\begin{cases}
\Bigl(\frac{x(x^2-3y^2)}{x^2+y^2},\frac{y(3x^2-y^2)}{x^2+y^2},
z\Bigr)&(x,y)\ne(0,0)\\
(0,0,z)&(x,y)=(0,0)
\end{cases}
$$
This mapping is 3-sheeted covering branching along $z$-axis but 
$df:\RR^3-\widehat{B}(f)\to\GL(3,\RR)$ is null homotopic, 
because $\pi_1(\SO(3))=\ZZ/2\ZZ$. 
\end{exam}

\end{document}